\documentclass{article}
\usepackage{latexsym, amsmath, amssymb}
\usepackage{amsmath, latexsym, amsfonts, amssymb, amsthm, amscd}
\usepackage{graphicx}
\usepackage{color}

\definecolor{pivert}{rgb}{0.0,0.9,0.4}

\newcommand{\R}{{\mathbb R}}
\newcommand{\Z}{\mathbb{Z}}

\newcommand{\T}{\mathcal{T}}
\newcommand{\Pcal}{\mathcal{P}}
\newcommand{\E}{\mathbb{E}}

\newcommand{\norm}[1]{\left \| #1 \right \|}
\newtheorem{thr}{Theorem}

\newtheorem{rem}{Remark}
\newtheorem{lem}{Lemma}

\begin{document}

\markboth{J.MacLaurin, J.Salhi & S.Toumi}{Mean field dynamics of a Wilson-Cowan Neuronal Network with Nonlinear Coupling Term } 

\title{Mean Field Dynamics of a  Wilson-Cowan Neuronal Network with Nonlinear Coupling Term}

\author{James MacLAURIN\\\small{School of Physics, University of Sydeney}\\Jamil SALHI\\ \small{NeuroMathcomp, Inria Sophia Antipolis;Lamsin-Enit,Tunis El Manar University}\\Salwa Toumi\\\small{Lamsin-Enit,Tunis El Manar University, Insat, Carthage University }}



\maketitle


\begin{abstract}
In this paper we prove the propagation of chaos property for an ensemble of interacting neurons subject to independent Brownian noise. The propagation of chaos property means that in the large network size limit, the neurons behave as if they are probabilistically independent. The model for the internal dynamics of the neurons is taken to be that of Wilson and Cowan, and we consider there to be multiple different populations. The synaptic connections are modelled with a nonlinear `electrical' model. The nonlinearity of the synaptic connections means that our model lies outside the scope of classical propagation of chaos results. We obtain the propagation of chaos result by taking advantage of the fact that the mean-field equations are Gaussian, which allows us to use Borell's Inequality to prove that its tails decay exponentially.
\end{abstract}



\section{Introduction}
Activity in the human brain results from the activity of billions of interacting neurons 
\cite{wilson1972excitatory,wilson1973mathematical,amari1977dynamics,ermentrout1998neural}. Ensembles of neurons can often produce
 coherent behaviour at macroscopic levels, and therefore it is desirable to use mathematical techniques to rigorously 
understand how the macroscopic description emerges out of the neuron-level description \cite{amit1985spin,benachour1998nonlinear,malrieu2003convergence,deco2008dynamic,budhiraja2012large,touboul2012limits,riedler2013laws,luccon2014mean}. The macroscopic phenomenon 
 we are interested in studying is propagation of chaos. This describes the situation where a large ensemble of neurons is heavily interconnected and subject to a lot of external noise, resulting in the neurons behaving randomly, with the activity of any two neurons probabilistically uncorrelated. These models are often described as `mean-field' models, because as a consequence of the uncorrelated behaviour of the neurons, their average activity is well-approximated by the mean activity of any particular neuron \cite{bressloff2012spatiotemporal}.

Mean-field models of neurons are typically a stochastic differential equation with three components:
 an internal dynamics term $\mathfrak{b}(X^j_t)$, an interaction term $ \frac{1}{N}\sum_{k=1}^N g(X^j_t,X^k_t)$ and a
 noise term $W^j_t$, so that the evolution of the state variable $X^j$ is given by the following
\begin{align}
dX^j_t &= \big(\mathfrak{b}(X^j_t) + \frac{1}{N}\sum_{k=1}^N g(X^j_t,X^k_t)\big)dt + dW^j_t\nonumber\\
&= \big(\mathfrak{b}(X^j_t) + \bar{g}(X^j_t,\hat{\mu}^N(X_t))\big)dt + dW^j_t,\label{eq mean field introduction}
\end{align}
where $\hat{\mu}^N(X_t) = \frac{1}{N}\sum_{k=1}^N \delta_{X^k_t}$ is the empirical measure and $\bar{g}(x,\mu) := \E^{\mu}[g(x,\cdot)]$. The above model is also often referred to as an interacting diffusion model. The noise terms $(W^j_t)$ are taken to be probabilistically independent. As long as the functions $\mathfrak{b}$ and $g$ are suitably well-behaved, the propagation of chaos result dictates that in the limit as $N\to \infty$, the law of $(X^j_t)_{t\in [0,T]}$ converges probabilistically to the law of some process $(\bar{X}_t)_{t\in [0,T]}$ which is the same for all $j$. This result was first proven by Sznitman \cite{sznitman1991topics}, and later by many others, including \cite{gartner1988mckean,renart2004mean,touboul2012noise,baladron2012mean,bossy2014clarification,inglis2014mean}. These works typically prove the convergence over fixed time intervals $[0,T]$, but sometimes the convergence is uniform in time \cite{salhi2015uniform}. 

There is a large variety of neuronal models, including those of Hodgkin-Huxley \cite{hodgkin1952quantitative}, FitzHugh-Nagumo 
 \cite{baladron2012mean} and Wilson-Cowan \cite{wilson1973mathematical,wilson1972excitatory,amari1977dynamics}. We employ the Wilson-Cowan model in this paper.  The Wilson-Cowan model was originally developed to model the interaction between excitatory and inhibitory neurons in a population \cite{wilson1972excitatory}. It is simpler than the other models: it was developed to capture some of the qualitative features of the dynamics of interacting neurons, without being too complicated. It is sufficient for our purpose because we are more interested in the effect of the synaptic connections on the behaviour of the neural network than the internal neuronal dynamics. 
 
Our model for the synaptic connections is that of an electrical synapse, which is more sophisticated than many other models of synaptic connections. They have previously been employed to model mean-field behaviour in neuronal networks in  \cite{baladron2012mean}. In electrical synapses, neurons are extremely close, so that current is able to directly flow through specialised structures called gap junctions \cite{kandel2000principles}. Gap junctions consist of precisely aligned channels in the presynaptic and postsynaptic neuron. These channels allow ions and other molecules to flow, thereby inducing an electric current. Often the current is induced by an action-potential, so that it is natural to model the current to be proportional to the difference between the activity of the presynaptic and postsynaptic neurons (since by Ohm's Law electrical current is proportional to the voltage difference) multiplied by the maximal conductance. We take the maximal conductance to be proportional to the activity of the presynaptic neuron, which we represent by a sigmoidal function. In conclusion, if $X_{pre}$ is the state variable of the pre-synaptic neuron and $X_{post}$ is the state variable of the post-synaptic neuron, then we take the total effect of the pre-synaptic neuron on the post-synaptic neuron to be proportional to 
 \begin{equation}\label{eqn synapse 1}
g(X_{post},X_{pre}) = C_g \big(X_{post} - X_{pre}\big)S(X_{pre}),
 \end{equation}
for some sigmoid $S$ and constant $C_g$ \cite{baladron2012mean}. 

Most proofs of propagation of chaos either assume that $g$ is globally Lipschitz in both of its variables, or that 
\begin{equation}
\sup_{x,y\in\R} |g(x,y)| < \infty. \label{eq sup g}
\end{equation}
The above condition unfortunately excludes the model of electrical synapses in \eqref{eqn synapse 1}. We must therefore prove the propagation of chaos result from first principles, adapting the proof of Sznitman  \cite{sznitman1991topics}. The proof works in our specific situation because, just as with \cite{touboul2012noise}, we can exploit the fact that the internal dynamics term $\mathfrak{b}$ is linear, which allows us to use the theory of Gaussian Processes to prove the existence and uniqueness of the mean-field limit. However our model of the synaptic interactions is more complicated than that of \cite{touboul2012noise} and we are forced to develop different methods. In more detail, we reduce the problem of existence and uniqueness of the mean-field limit to an existence and uniqueness problem for the ordinary differential equations governing the mean and variance. The Gaussian nature of the mean-field limit also allows us to use Borell's Inequality to bound the tails of $\bar{X}_t^j$. In the classical theory, it has been proven that the rate of convergence of the Wasserstein Distance is $O(N^{-\frac{1}{2}})$: however in our model we do not obtain a specific rate (it is probably slower than  $O(N^{-\frac{1}{2}})$). It should be noted that \cite{bossy2014clarification}, like us, also do not assume \eqref{eq sup g}.

Another feature of our work is that we consider a multipopulation model. Classically, multipopulations models are typically composed of activator and inhibitor neurons, but in our work the number of populations may be arbitrarily high. The ratio of the number of neurons in each population is kept fixed. Mean-field multipopulation models have already been studied in the original work of Wilson and Cowan \cite{wilson1972excitatory}, but also in that of \cite{faugeras2009constructive}, \cite{bressloff2010metastable}, \cite{touboul2012noise} and \cite{baladron2012mean}, amongst others. Multipopulation models can admit much richer dynamics: as an example it is well-known that activator-inhibitor models can admit limit-cycle solutions \cite{touboul2012noise}. 

One of the basic reasons to obtain mean-field equations is that the unreduced system is too complicated to study. However even the mean-field equations are frequently difficult to study, because the drift term is a function of the 
law $\bar{\mu}_t$ of the limit; so that often there is not an elegant way of characterising the solution.
 However in our specific case, the mean-field equations are Gaussian, and so their behaviour can be further reduced to the
 ODEs governing the mean and variance. The strength of our paper is that we have reduced the seemingly very complicated
 question of studying the dynamics of a large system of interacting neurons subject to noise, to the much-more 
tractable problem of studying the dynamics of a set of coupled ODEs. We note that \cite{touboul2012noise} have already 
studied a similar problem. They obtain mean-field equations for a system of interacting Wilson-Cowan neurons. 
They also reduce the mean-field equations to a set of ODEs governing the mean and variance, and study the bifurcations 
of these ODEs. Our work is different because our interaction term is that of an electrical synapse (our interaction term is not bounded).

The structure of this paper is as follows. In Section \ref{model} we outline a general multipopulation model of interacting 
Wilson-Cowan neurons, coupled by electrical synapses and subject to noise. In Section \ref{theory} we introduce the
 mean-field equations for our system and prove the propagation of chaos property, namely that the finite-size system 
converges to the mean-field limit as the number of neurons asymptotes to infinity. In Section \ref{numeric} we give an example 
of a neural network which contains two populations, and perform a numerical simulation of the mean-field equation.

\section{Mathematical Model And Assumptions}\label{model}

In this section we outline our model and our assumptions. As we have already noted in the introduction, our model is that of a multipopulation model of Wilson-Cowan neurons, connected by electrical synapses and subject to Brownian noise. We introduce and explain our model further below. However before we do this, we must introduce some notation.
 
We assume that the number of populations $P$ is fixed, and that 
the ratio of neurons in each population is fixed. We fix the integer constants $s_1,\ldots s_P > 0$ 
which give the relative proportion of neurons in each population. We organise the neurons into $(2n+1)$ groups,
 with each group containing neurons from all of the populations. In the mean-field limit as $n\to \infty$, we will 
find that the behaviour of a single group of neurons is characteristic of the behaviour of the entire ensemble. 
Within each group $i \in [-n,n]$, we specify that there are $2s_\alpha+1$ neurons from population $\alpha$. We accordingly index
 each neuron by $(i,\alpha,p)$, for $i\in [-n,n]$, $\alpha\in [1,P]$ and $p\in [-s_\alpha,s_\alpha]$. Observe that the total number of neurons across all of the groups in population $\alpha \in [1,P]$ is $(2s_\alpha+1) \times (2n+1)$. Our goal is to asymptote $n\to \infty$, while keeping the other parameters fixed. 

We work in a complete probability space $(\Omega,\mathcal{F},\mathbb{P})$, endowed with a filtration $(\mathcal{F}_t)$ and satisfying the usual conditions. The set of all probability measures on some topological space $\mathcal{X}$ (endowed with the Borelian sigma-algebra) is written as $\Pcal(\mathcal{X})$. The state variable for any single neuron is assumed to be in $\R$. Let $\bar{s}= \sum_{\alpha=1}^P (2s_\alpha + 1)$. We write $\T := \R^{\bar{s}}$ to be the state space for a single group of populations, using the indexing (for $i\in [-n,n]$): $X_t^i \in \T$, where $X_t^i := (X_t^{i,\alpha,p})_{\alpha\in [1,P],p\in [-s_\alpha,s_\alpha]}$. In general, when $Y \in \R^q$ for some $q>0$, $\norm{Y} = \sup_{r\in [1,q]}\left| Y^r \right|$. Throughout this article, we work over the time interval $[0,T]$, for some fixed $T>0$.

Let $(W^{i,\alpha,p})_{i\in [-n,n],\alpha\in [1,P],p\in [-s_\alpha,s_\alpha]}$ be mutually independent Wiener Processes on $[0,T]$. The  general evolution equation of our neural network is of the form
\begin{multline}\label{networkmodel}
 dX^{i,\alpha,p}_t = -\frac{1}{\tau}X^{i,\alpha,p}_t+I^{\alpha}(t)+\\ \frac{1}{2n+1} \sum_{j=-n}^{n}\sum_{\beta=1}^P
\sum_{q=-s_\beta}^{s_\beta} \bar{J}^{\alpha\beta}\left(X_t^{i,\alpha,p}-X_t^{j,\beta,q}\right) S(X_t^{j,\beta,q})dt + \sigma^\alpha dW^{i,\alpha,p}_t,
\end{multline}
with $X_0 = X_{ini}$ (a constant). We will often neglect the summation limits if they are clear from the context. Here $I^{\alpha}\in \mathcal{C}(\R)$ is  the deterministic input current; it is assumed to be continuous and bounded and homogeneous across populations.  $\tau$ and $\sigma^\alpha$ are positive constants: $\tau$ gives the magnitude of the decay due to the internal dynamics, and $\sigma^{\alpha}$ the relative strength of the noise. 

We have already noted in the introduction that we model the synaptic connections as electrical. It can be seen that the strength of connection from a presynaptic neuron in population $\alpha$ to a postsynaptic neuron in population $\beta$ is scaled by the constant $\bar{J}^{\alpha\beta} \in \R$. $S$ is a sigmoid function: monotonically increasing, twice differentiable everywhere and satisfying $0 < S(x) < 1$. We assume that there exists a positive constant $C_S$ such that
\begin{align*}
\sup_{x\in\R}\left| S'(x)x \right| &\leq C_S \text{ and that } \\
\sup_{x\in \R}\left|x\frac{d^2 S}{dx^2}\right|  &< \infty.
\end{align*}
Write $\sigma_{min} = \rm{sup}_{\alpha\in [1,P]} \sigma^\alpha $ (this is assumed to be strictly greater than zero), $\sigma_{max} = \rm{sup}_{\alpha \in [1,P]} \sigma^\alpha$ and $\bar{J}^{max} = \sup_{\alpha,\beta \in [1,P]}\left|\bar{J}^{\alpha\beta}\right|$.
The empirical measure is defined to be
\begin{equation}
\hat{\mu}^n(X) := \frac{1}{2n+1}\sum_{i=-n}^n \delta_{X^i} \in \mathcal{P}(\mathcal{C}([0,T], \T)).
\end{equation}
 The evolution equation (1) may thus be written as 
\begin{multline}\label{eq:membraneevolution}
dX^{i,\alpha,p}_t = \left[-\frac{1}{\tau}X^{i,\alpha,p}_t + I^\alpha(t)\right]dt + \sigma^\alpha dW^{i,\alpha,p}_t + \\ \left[\int_{\T}\sum_{\beta=1}^P
\sum_{q=-s_\beta}^{s_\beta}  \bar{J}^{\alpha\beta}\left(X_t^{i,\alpha,p}-Y_t^{\beta,q}\right) S(Y_t^{\beta,q})d\hat{\mu}^n(X)(Y)\right]dt,
\end{multline}
where the integration in the second line is performed with respect to the empirical measure. The following theorem guarantees strong existence and uniqueness to the evolution equations.
\begin{thr}\label{theorem1}
There exists a unique strong solution to (\ref{networkmodel}).
\end{thr}
\begin{proof}
This follows from a direct application of \cite[Theorem 3.5]{mao2007stochastic}.
\end{proof}
\section{Mean-field description and propagation of chaos}\label{theory}

In this section we prove the propagation of chaos property. There are two key theorems. In the first theorem
 (Theorem \ref{Thm Mean Field Existence and Uniqueness}), we prove that there exists a unique solution to the mean-field 
equation given just below in \eqref{meanelec1}. In the second theorem, Theorem \ref{theorem3}, we prove that the expectation of the distance between the solution of the mean-field equation and the solution of a single group $X^i$ (for some fixed $i$)
 from \eqref{networkmodel} goes to zero as $n\to \infty$. We could not use classical propagation of chaos results 
such as those in \cite{sznitman1991topics} because the interaction due to the electrical synapses is not uniformly Lipschitz.
 Neither does our model fit in the framework of more recent results such as those in \cite{bolley2013uniform}. As we note in Remark \ref{Rem Wasserstein}, this result implies that the $1^{st}$ Wasserstein distance between the probability laws goes to zero. 

Let $(W^{\alpha,p})_{\alpha \in [1,P], p\in [-s_\alpha,s_\alpha]}$ be independent Wiener processes. Define $\bar{X}$ to be the solution of 
\begin{multline}\label{meanelec1}
d\bar{X}_t^{\alpha,p} =  \sigma^\alpha dW_t^{\alpha,p} + \\ \left[ -\frac{\bar{X}_t^{\alpha,p}}{\tau} + I^\alpha(t) +\int_{\T}\sum_{\beta=1}^P\sum_{q=-s_\beta}^{s_\beta} 
 \bar{J}^{\alpha\beta}\left(\bar{X}_t^{\alpha,p}-Y^{\beta,q}\right) S(Y^{\beta,q})
d\bar{\mu}_t(Y)\right] dt,
\end{multline}
for $\alpha\in [1,P]$ and $\bar{X}_0 := X_{ini}\in \T$. Here $\bar{\mu} \in \mathcal{P}(\mathcal{C}([0,T],\T))$ is the law of $\bar{X}$, with time marginal $\bar{\mu}_t \in \mathcal{P}(\T)$. The existence and uniqueness of the solution $\bar{X}$ is proved in the following section. We note that by symmetry, the law of $\bar{X}^{\alpha,p}$ is the same as the law of $\bar{X}^{\alpha,q}$ for any $p,q \in [-s_\alpha,s_\alpha]$.

Our first main theorem is the following.
\begin{thr}\label{Thm Mean Field Existence and Uniqueness}
There exists a unique strong solution $\bar{X}$ (with law \newline$\bar{\mu} \in \Pcal( \mathcal{C}([0,T], \T))$) to (\ref{meanelec1}) which satisfies
\begin{equation}
 \E \left[ \sup_{t\in [0,T]}  \norm{\bar{X}_t}^2\right] < \infty.\label{eqn norm bound}
\end{equation}
\end{thr}
The propagation of chaos (our second main result) is the following.
\begin{thr}\label{theorem3}
\[
\lim_{n\to\infty}\mathbb{E}\left[\sup_{t\in [0,T]}\norm{X_t^1-\bar{X}^1_t}\right] = 0.
\]
\end{thr}
\begin{rem}\label{Rem Wasserstein}
This result implies that
\begin{equation}
\lim_{n\to\infty}d^W(\mu^1,\bar{\mu})= 0,
\end{equation}
where $\mu^1 \in \mathcal{P}\big(\mathcal{C}([0,T],\mathcal{T})\big)$ is the probability law of \newline $X^1 := (X^{1,\alpha,p})_{\alpha\in[1,P],p\in [-s_\alpha,s_\alpha]}$ in \eqref{networkmodel}. Here $d^W$ is the $1^{st}$ Wasserstein Distance, a metric on $ \mathcal{P}\big(\mathcal{C}([0,T],\mathcal{T})\big)$ given by
\begin{equation}
d^W\big(\nu,\gamma) = \inf_{\zeta\in\Gamma(\nu,\gamma)}\mathbb{E}^{\zeta}\left[ \sup_{t\in [0,T]}\norm{x(t)-y(t)}\right],
\end{equation}
where $\Gamma(\nu,\gamma)$ is the set of all probability measures on $\mathcal{C}([0,T],\mathcal{T}) \times \mathcal{C}([0,T],\mathcal{T})$ with marginals $\nu$ and $\gamma$ which are (respectively) the laws of the processes $x(t)$ and $y(t)$.
\end{rem}

\begin{rem}
We have assumed for notational convenience throughout this paper that the initial condition is a constant. In fact this paper easily generalises to the situation where the initial condition of each $X^i$ is Gaussian, as long as $X^i_0$ is independent of $X^j_0$ for $i \neq j$.
\end{rem}
\subsection{Existence and Uniqueness of Solution to the Mean Field Equations}
Classically, \cite{sznitman1991topics} showed existence and uniqueness to the mean-field equations by using a fixed-point argument. It is not immediately clear how one might adapt this method to our context because the interaction terms are not uniformly Lipschitz. Instead, we proceed by noting that any solution of the mean-field equations must be Gaussian. We may then reduce the problem of existence and uniqueness of a solution to \eqref{meanelec1} to the problem of existence and uniqueness of a solution to the ordinary differential equation governing the mean and variance of the Gaussian solution.

\begin{proof}[Proof of Theorem \ref{Thm Mean Field Existence and Uniqueness}]
The existence and uniqueness is a direct consequence of Lemmas \ref{lemma equivalence} and \ref{Lem Existence} (below). 
 Equation \eqref{eqn norm bound} follows from the fact that the solution $(\bar{X}_t)_{t\in [0,T]}$ is continuous and Gaussian 
with bounded mean
 and variance and  Borell's Inequality (see \cite[Theorem 2.1]{adler1990introduction}).
\end{proof}

We start by proving Lemma \ref{lemma equivalence}, i.e. that the existence and uniqueness to the SDE reduces to the question of existence and uniqueness to the ODEs governing the mean and variance.

We must first introduce some new notation. Using the fact that $\big(\bar{X}^{\alpha,p}\big)_{p\in [-s_\alpha,s_\alpha]}$ are identically distributed, we may write (\ref{meanelec1}) as 
\begin{equation}\label{meanelec2}
d\bar{X}^{\alpha,p}_t=\left(F^{\alpha}(\bar{\mu}_t) \bar{X}_t^{\alpha,p}+f^{\alpha}(t,\bar{\mu}_t)\right)dt + \sigma^\alpha dW_t^{\alpha,p}
\end{equation}
where for some $\gamma\in \mathcal{P}(\T)$,
$$F^{\alpha}(\gamma):=-\frac{1}{\tau}+ \sum_{\beta=1}^p(2s_\beta+1)\int_{\T}\bar{J}^{\alpha\beta}
S\big(y^{\beta,0}\big)d\gamma(y)$$
and
$$f^{\alpha}(t,\gamma):=I^\alpha(t)- \sum_{\beta=1}^P(2s_\beta+1)\int_{\T}\bar{J}^{\alpha\beta}
y_t^{\beta,0}S(y^{\beta,0})d\gamma(y).$$

Let $\bar{X}_t := (\bar{X}^{\alpha,p}_t)_{\alpha\in [1,P],p\in [-s_\alpha,s_\alpha]}$.

Let $(\xi_t)_{t\in [0,T]} \subset \mathcal{P}(\T)$ such that $t \mapsto F(\xi_t)$ and $t \mapsto f(t,\xi_t)$ are continuous in $t$. In particular, if there exists $\xi \in \mathcal{P}(\mathcal{C}([0,T],\T))$ with marginals $(\xi_t)_{t\in [0,T]}$ such that $\E^{\xi(y)}\left[\sup_{t\in [0,T]} \norm{y_t}^2\right] < \infty$, then this property would be satisfied. To see this, let $t_{(m)}\to t$ as $m\to \infty$. Then since $y_{t_{(m)}} \to y_t$ for any $y \in \mathcal{C}([0,T],\T)$, by the dominated convergence theorem for any $\alpha,\beta \in [1,P]$ and $p \in [-s_\alpha,s_\alpha]$ and $q\in [-s_\beta,s_\beta]$,
$$
\E^{\xi(y)}\left[ y_{t_{(m)}}^{\alpha,p}y_{t_{(m)}}^{\beta,q} - y_t^{\alpha,p}y_t^{\beta,q}\right] \to 0,\text{ and } \E^{\xi(y)}\left[ y^{\alpha,p}_{t_{(m)}} - y^{\alpha,p}_t\right] \to 0, 
$$
as $m\to \infty$. In any case, the solution of the following stochastic differential equation is Gaussian (see for instance \cite[Section 5.6]{karatzas2012brownian}),

\begin{equation}\label{eq:Gaussian21}
dY^{\alpha,p}_t=\left[F^{\alpha}(\xi_t) Y_t^{\alpha,p}+f^{\alpha}(t,\xi_t)\right]dt + \sigma^\alpha dW_t^{\alpha,p}.
\end{equation}

It is clear that the drift of (\ref{eq:Gaussian21}) is linear in $Y_t$, which means that there exists a unique Gaussian Solution \cite[Section 5.6]{karatzas2012brownian}). To see this, notice that we can write the solution as
\[
Y^{\alpha,p}_t = \phi_\alpha(t)\bigg(X_{ini}^{\alpha,p} + \int_0^t \phi_\alpha^{-1}(s) f^\alpha(s,\xi_s)ds + \int_0^t \phi_\alpha^{-1}(s)\sigma^\alpha dW_s^{\alpha,p}\bigg),
\]
where $\phi_\alpha$ satisfies $\frac{d\phi_\alpha(t)}{dt} = F^{\alpha}(\xi_t)\phi_\alpha(t)$ and $\phi_\alpha(0) = 1$. One may see that the marginals of the process $\int_0^t \phi_\alpha^{-1}(s)\sigma^\alpha dW_s^{\alpha,p}$ are Gaussian by approximating the integrand using simple functions and then taking limits (see \cite[Section 5.6]{karatzas2012brownian}). This means that the marginals of $(Y^{\alpha,p}_t)_{t\in [0,T]}$ are also Gaussian. 

As a particular application of the above theory, any solution of \eqref{meanelec2} must be Gaussian (since it comes from substituting $\xi_t = \bar{\mu}_t$,  the law of the mean field process $\in \mathcal{P}(\mathcal{C}([0,T],\T))$, into \eqref{eq:Gaussian21}) and can be characterized by the evolution of its mean and covariance.

 Assuming for the moment that there is a unique solution, let $m(t) := (m^{\alpha}(t))_{\alpha \in [1,P]}$ be the mean,
 where for any $p\in [-s_\alpha,s_\alpha]$, $m^{\alpha}(t) = \E[\bar{X}_t^{\alpha,p}]$ and $Q(t) := (Q_{pq}^{\alpha\beta}(t))$
 the covariance, where for any $p\in [-s_\alpha,s_\alpha]$ and $q\in [-s_\beta,s_\beta]$, $Q_{pq}^{\alpha\beta}(t) = \E\big[(\bar{X}^{\alpha,p}_t - m^{\alpha}(t))(\bar{X}^{\beta,q}_t -
 m^{\beta}(t))\big]$. 
 
It follows from \eqref{eq:Gaussian21} that the mean and variance must satisfy the differential system
 (see for instance \cite[Section 5.6]{karatzas2012brownian}). For all $\alpha,\beta\in [1,P]$, $p\in [-s_\alpha,s_\alpha]$ and $q\in [-s_\beta,s_\beta]$
\begin{align}
\frac{d}{dt}m^{\alpha}(t) =& F^\alpha(\bar{\mu}_t)m^{\alpha}(t) + f^\alpha(t,\bar{\mu}_t) \label{eqn m evolution} \\
\frac{d}{dt}Q_{pq}^{\alpha\beta}(t) =& \big(F^\alpha(\bar{\mu}_t) + F^{\beta}(\bar{\mu}_t)\big)Q_{pq}^{\alpha\beta}(t) + \delta(p,q)\delta(\alpha,\beta)\big(\sigma^\alpha\big)^2
\end{align}
with $m(0) = X_{ini}$ and $Q_{pq}^{\alpha\beta}(0) = 0$. 
It may be seen that $Q^{\alpha\beta}_{pq}(t) = 0$ if $\alpha \neq \beta$ or $p\neq q$, and that
 $Q^{\alpha\alpha}_{pp}(t) = Q^{\alpha\alpha}_{qq}(t)$ for all $p,q\in [-s_\alpha,s_\alpha]$. Thus we only need to determine $\big(Q^{\alpha\alpha}_{00}(t)\big)_{\alpha\in [1,P]}$.

For $w > 0$ and $r\in \R$, let $\rho(r,w,y) = (2\pi w)^{-\frac{1}{2}}\exp\left( - \frac{(y-r)^2}{2w}\right)$ be the Gaussian density. It may be seen that if $Q_{00}^{\beta\beta}(t) > 0$ for all $\beta \in [1,P]$, then
\begin{multline}\label{eqn F mut}
F^{\alpha}(\bar{\mu}_t)=-\frac{1}{\tau}+ \sum_{\beta=1}^P(2s_\beta + 1)\bar{J}^{\alpha\beta}\int_{\R}S(y)\rho(m^{\beta}(t),Q_{00}^{\beta\beta}(t),y)dy := \bar{F}^{\alpha}(m(t),Q(t))
\end{multline}
and
\begin{multline}\label{eqn f mut}
f^{\alpha}(t,\bar{\mu}_t)=I^\alpha(t)- \sum_{\beta=1}^P(2s_\beta + 1)\bar{J}^{\alpha\beta}
\int_{\R}y S(y)\rho(m^{\beta}(t),Q_{00}^{\beta\beta}(t),y)dy 
:= \bar{f}^{\alpha}(t,m(t),Q(t)).
\end{multline}
If $Q_{00}^{\beta\beta}(t) = 0$ in the above expressions for (possibly several) $\beta \in [1,P]$, then we replace the integrals in the above summations by  
\begin{align*}
\int_{\R}S(y)\rho(m^{\beta}(t),Q_{00}^{\beta\beta}(t),y)dy &\to S\big(m^{\beta}(t)\big)\\
\int_{\R}y S(y)\rho(m^{\beta}(t),Q_{00}^{\beta\beta}(t),y)dy &\to m^{\beta}(t)S\big(m^{\beta}(t)\big).
\end{align*}
Thus we may consider $f$ and $F$ to be functions $\bar{f},\bar{F}$ of $m(t)$ and $Q(t)$, which are well-defined and continuous for $Q_{00}^{\beta\beta}(t) \geq 0$ and $m^{\beta}(t)\in\R$.
\begin{lem}\label{lemma equivalence}
There exists a unique solution $\bar{\mu}$ of the SDE \eqref{meanelec2} if and only if there exists a unique solution (for all time) of the following system of ordinary differential equations
\begin{align}
\frac{d}{dt}m^{\alpha}(t) =& \bar{F}^\alpha(m(t),Q(t))m^\alpha(t) + \bar{f}^\alpha(t,m(t),Q(t)) \label{eqn mbar} \\
\frac{d}{dt}Q_{00}^{\alpha\alpha}(t) =& 2\bar{F}^\alpha(m(t),Q(t))Q_{00}^{\alpha\alpha}(t) + (\sigma^\alpha)^2, \label{eqn Pbar}
\end{align}
for all $\alpha \in [1,P]$ with $m(0) = \bar{X}_{ini}$ and $Q^{\alpha\alpha}_{00}(0) = 0$.
\end{lem}
\begin{proof}
We have already seen the necessity of \eqref{eqn mbar}-\eqref{eqn Pbar}. For the sufficiency, suppose that we have a solution $(m,Q)$ to \eqref{eqn mbar}-\eqref{eqn Pbar}. For each $t\geq 0$, we then define $\xi_t \in \mathcal{P}\big(\T\big)$ to be a Gaussian distribution with mean $m(t)$ and variance $Q(t)$. More precisely, we define $\mathbb{E}^{\xi_t(y)}[y^{\alpha,p}] := m^\alpha(t)$, for all $\alpha\in [1,P]$ and $p\in [-s_\alpha,s_\alpha]$, $\mathbb{E}^{\xi_t(y)}\big[(y^{\alpha,p} - m^\alpha(t))^2\big] := Q_{00}^{\alpha\alpha}(t)$ and $\mathbb{E}^{\xi_t(y)}\big[(y^{\alpha,p}-m^\alpha(t))(y^{\beta,q}-m^\beta(t))\big] := 0$ if $\alpha\neq\beta$ or $p\neq q$. We note that  this is a well-defined covariance function which is positive definite, because $Q_{00}^{\alpha\alpha}(t) \geq 0$ for all $t\in [0,T]$ (since if $Q_{00}^{\alpha\alpha}(t) = 0$ then $\frac{d}{dt}Q_{00}^{\alpha\alpha}(t) > 0$). 

We substitute this definition of $\xi_t$ into  \eqref{eq:Gaussian21}. Using this definition of $\xi_t$, \eqref{eq:Gaussian21} becomes a linear SDE with coefficients which are continuous in time, for which there is a unique strong solution, as noted in \cite[Section 5.6]{karatzas2012brownian}. This solution is Gaussian, and we write its law as $\bar{\mu} \in \mathcal{P}(\mathcal{C}([0,T],\T))$. The marginal $\bar{\mu}_t$ has mean $m(t)$ and variance $Q(t)$. This means that the strong solution must also satisfy \eqref{meanelec2}.
\end{proof}

Let $\mathcal{R} = (m^{\alpha},Q_{00}^{\alpha\alpha})_{\alpha \in [1,P]}$ be the solution space of \eqref{eqn mbar}-\eqref{eqn Pbar}, with $Q_{00}^{\alpha\alpha}\geq 0$. We have left out of the definition of $\mathcal{R}$ the covariances $Q^{\alpha\beta}_{pq}$ if $\alpha\neq\beta$ or $p\neq q$ because they are trivially zero. We use the supremum norms $\norm{m} = \sup_{\alpha\in [1,P]}|m^{\alpha}|$ and $$\norm{Q}=\sup_{\alpha\in [1,P]} |Q_{00}^{\alpha\alpha}|.$$

\begin{lem}\label{Lem Existence}
There exists a unique solution $(m,Q)\in \mathcal{C}([0,T],\mathcal{R})$ to \eqref{eqn mbar} and \eqref{eqn Pbar}.
\end{lem}
\begin{proof}
The right hand sides of  \eqref{eqn mbar} 
and \eqref{eqn Pbar} satisfy a locally Lipschitz condition thanks to Lemma \ref{lem f F globally Lipschitz} below
(note also the definitions in \eqref{eqn F mut} and \eqref{eqn f mut}). This locally Lipschitz property extends to the case $Q_{00}^{\alpha\alpha} = 0$ (for any $\alpha\in [1,P]$) thanks to the continuity of $\bar{f}$ and $\bar{F}$ as functions of $Q(t)$.

We proceed inductively to prove that there exists a unique solution over the time interval $[0,T]$. Suppose that we have a unique solution $(m(t),Q(t))$ for $t\in [0,\gamma]$ for some $\gamma \in [0,T)$.  Write $m_\gamma := m(\gamma)$ and $Q_\gamma := Q(\gamma)$. Let $\epsilon = \max\lbrace \norm{m_\gamma},\norm{Q_\gamma}\rbrace$ and 
\begin{equation*}
B = \lbrace (m,Q)\in\mathcal{R} : \norm{m - m_\gamma} \leq \max(\epsilon,1), \norm{Q- Q_\gamma} \leq \max(\epsilon,1)\rbrace. 
\end{equation*}
It follows from this that if $(m,Q) \in B$, then
\begin{align}
\norm{m} &\leq \norm{m_\gamma} + \epsilon + 1 \leq 2\epsilon + 1 \label{eq m bound 1}\\
\norm{Q} &\leq \norm{Q_\gamma} + \epsilon + 1 \leq 2\epsilon+ 1.\label{eq Q bound 1}
\end{align}
From Lemma \ref{Lem Lipschitz Growth} (below) and \eqref{eq m bound 1}-\eqref{eq Q bound 1}, we observe that for all $\alpha \in [1,P]$ and $t\in [0,T]$,
\begin{align}
\sup_{(m,Q)\in B} \left|\bar{F}^\alpha(m,Q)m^\alpha + \bar{f}^\alpha(t,m,Q)\right| \leq A (4\epsilon + 3) \\
\sup_{(m,Q)\in B} \left|2\bar{F}^\alpha(m,Q)Q_{00}^{\alpha\alpha} + (\sigma^\alpha)^2\right| \leq A(2\epsilon +2).
\end{align}
Let
\begin{align}
\upsilon = \inf_{\epsilon \geq 0}\left\lbrace T-\gamma,\frac{\max(\epsilon,1)}{A(4\epsilon + 3)}, \frac{\max(\epsilon,1)}{A(2\epsilon + 2)}\right\rbrace.
\end{align}
Observe that $\upsilon  > 0$. By the Cauchy-Peano Theorem \cite[Theorem 2.1]{hartman1964ordinary}, there exists a solution over the time interval $[\gamma,\gamma +\upsilon]$ to \eqref{eqn mbar} and \eqref{eqn Pbar} with the initial condition such that $(m(\gamma),Q(\gamma)) = (m_\gamma,Q_\gamma)$.
This solution is unique because the right hand side of \eqref{eqn mbar}-\eqref{eqn Pbar} is locally Lipschitz over $B$. Since $\upsilon$ is independent of $\epsilon$, we may continue iteratively until we find that there exists a unique solution for all $t\in [0,T]$.
\end{proof}

\begin{lem}\label{lem f F globally Lipschitz}
There exists a constant $K$ such that for all $\mathfrak{m},\mathfrak{m}_2 \in \R$ and all $\mathfrak{V},\mathfrak{V}_2> 0$,
\begin{align}
\left| \int_{\R}S(y)\left(\rho(\mathfrak{m},\mathfrak{V},y) - \rho(\mathfrak{m}_2,\mathfrak{V}_2,y)\right) dy\right| \leq K \left(|\mathfrak{V} - \mathfrak{V}_2|+|\mathfrak{m}-\mathfrak{m}_2|\right),\label{assmp Holder 1}\\
\left| \int_{\R}yS(y)\left(\rho(\mathfrak{m},\mathfrak{V},y) - \rho(\mathfrak{m}_2,\mathfrak{V}_2,y)\right) dy\right| \leq K \left(|\mathfrak{V} - \mathfrak{V}_2|+|\mathfrak{m}-\mathfrak{m}_2|\right).\label{assmp Holder 2}
\end{align}
Here $\rho(r,w,y) := (2\pi w)^{-\frac{1}{2}}\exp\left( - \frac{(y-r)^2}{2w}\right)$ is a Gaussian kernel.
\end{lem}
\begin{proof}
We prove the second of the above identities (the first follows very similarly). Assume without loss of generality that $\mathfrak{V}_2 \geq \mathfrak{V}$. 
Define $y_t$ to satisfy the following 1-dimensional SDE. Let $y_0$ be Gaussian, with mean $\mathfrak{m}$ and variance $\mathfrak{V}$, and $(W_t)$ a standard Wiener Process, such that  
\[
dy_t = (\mathfrak{m}_2 - \mathfrak{m})dt + \sqrt{\mathfrak{V}_2 - \mathfrak{V}}dW_t.
\]
Using standard theory we see that $y_t$ possesses a unique strong solution possessing Gaussian marginals. For any $t\geq 0$, $y_t$ has mean $\mathfrak{m} + t(\mathfrak{m}_2 - \mathfrak{m})$ and variance $\mathfrak{V} + t(\mathfrak{V}_2 - \mathfrak{V})$. Using Ito's Lemma,
\begin{multline*}
\int_{\R}yS(y)\left(\rho(\mathfrak{m}_2,\mathfrak{V}_2,y)-\rho(\mathfrak{m},\mathfrak{V},y)\right) dy = \E\left[ y_1S(y_1) - y_0S(y_0)\right] \\= \E\left[ \int_0^1\left((\mathfrak{m}_2 - \mathfrak{m})\frac{d}{dy_t}(y_tS(y_t)) + \frac{1}{2}(\mathfrak{V}_2 - \mathfrak{V})\frac{d^2}{dy_t^2}(y_t S(y_t))\right)dt\right] \\
= (\mathfrak{m}_2 - \mathfrak{m})\E\left[ \int_0^1\frac{d}{dy_t}(y_tS(y_t))dt\right] + \frac{1}{2}(\mathfrak{V}_2 - \mathfrak{V})\E\left[\int_0^1\frac{d^2}{dy_t^2}(y_t S(y_t))dt\right].
\end{multline*} 
Now, by Lemma \ref{Lem Globally Lipschitz} (below) and the assumptions on $S$, there is a constant $K_2$ such that $\sup_{y \in \R} | \frac{d}{dy}\big(yS(y)\big)| < K_2$ and $\sup_{y \in \R} \left| \frac{d^2}{dy^2}\big( yS(y)\big)\right| < 2K_2$. We thus see that
\[
\left|\int_{\R}yS(y)\left(\rho(\mathfrak{m},\mathfrak{V},y) - \rho(\mathfrak{m}_2,\mathfrak{V}_2,y)\right) dy\right| \leq K_2\left(\left|\mathfrak{m}_2-\mathfrak{m}\right| + \left|\mathfrak{V}_2 - \mathfrak{V}\right|\right).
\]
\end{proof}

\begin{lem}\label{Lem Lipschitz Growth}
There exists $A > 0$, such that for all $(m,Q) \in \mathcal{R}$, and all $\alpha \in [1,P]$,
\begin{align*}
\norm{\bar{F}^\alpha(m,Q)m^\alpha + \bar{f}^\alpha(t,m,Q)} &\leq A\left(\norm{m} + \norm{Q} + 1\right) \\ 
\norm{2\bar{F}^\alpha(m,Q)Q^{\alpha\alpha} + (\sigma^\alpha)^2} &\leq A\left(\norm{Q}+1\right).
\end{align*}
\end{lem}
\begin{proof}
It is clear that the bounds hold if $Q_{00}^{\alpha_i\alpha_i} = 0$ for some $\lbrace \alpha_i\rbrace \subset [1,P]$ (since the Gaussian Integration converges to a $\delta$-function). We therefore assume that for all $\alpha \in [1,P]$, $Q^{\alpha\alpha}_{00} > 0$. We note the following bounds (making use of the triangle inequality)
\begin{align*}
\left| \bar{f}^{\alpha}(t,m,Q)\right| &\leq |I^\alpha(t)| + \sum_{\beta=1}^P(2s_\beta + 1)\bar{J}^{max} \int_{\R} \big|y\big| (2\pi Q_{00}^{\beta\beta})^{-\frac{1}{2}}\exp\left( - \frac{(y-m^{\beta})^2}{2Q_{00}^{\beta\beta}}\right)dy\\
&\leq \norm{I(t)} + \\  \sum_{\beta=1}^P(2s_\beta + 1)\bar{J}^{max}&\left( |m^{\beta}| + \int_{\R} | y-m^{\beta} |(2\pi Q_{00}^{\beta\beta})^{-\frac{1}{2}}\exp\left( - \frac{(y-m^{\beta})^2}{2Q_{00}^{\beta\beta}}\right)dy\right)\\
&= \norm{I(t)} + \sum_{\beta=1}^P(2s_\beta + 1)\bar{J}^{max}\left( |m^{\beta}|  + \frac{\sqrt{2}}{\sqrt{\pi}}\sqrt{Q_{00}^{\beta\beta}}\right),
\end{align*}
using a standard Gaussian identity. Hence, recalling that $\bar{s} = \sum_{\beta=1}^P (2s_\beta+1)$,
\begin{align}\label{eq f lipschitz}
\norm{\bar{f}(t,m,Q)} \leq \norm{I(t)} + \bar{s}\bar{J}^{max}\left(\norm{m} + \norm{Q} + 1\right).
\end{align}
We also observe that
\begin{equation*}\label{eq F lipschitz}
\left| \bar{F}^{\alpha}(m,Q)\right| \leq \frac{1}{\tau} + \bar{J}^{max}\bar{s},
\end{equation*}
from which the second bound quickly follows.
\end{proof}
\begin{lem}\label{Lem Globally Lipschitz}
The function $x\mapsto xS(x)$ is globally Lipschitz in $x$.
\end{lem}
\begin{proof}
We have that $\frac{d}{dx}\big(xS(x)\big) = xS'(x) + S(x)$. This is bounded (by assumption on $S$). Therefore by the Mean Value Theorem, $xS(x)$ must be globally Lipschitz.
\end{proof}
\subsection{Propagation of Chaos}
In this section we prove the propagation of chaos property. This is stated in Theorem \ref{theorem3}, and amounts to
 proving that the first Wasserstein distance between the laws of $\bar{X}^j$ and $X^j$ (considered as random variables on $\mathcal{C}([0,T],\T)$ goes to zero as $n\to \infty$. Here, for $j\in \Z$, 
\begin{multline}\label{meanelec3}
d\bar{X}_t^{j,\alpha,p} = -\frac{\bar{X}_t^{j,\alpha,p}}{\tau} + I^\alpha(t) + \int_{\T}\sum_{\beta=1}^P\sum_{q=-s_\beta}^{s_\beta} 
 \bar{J}^{\alpha\beta}\left(\bar{X}_t^{j,\alpha,p}-y^{\beta,q}\right) S(y^{\beta,q})
\bar{\mu}_t(dy) dt\\ + \sigma^\alpha dW_t^{j,\alpha,p}.
\end{multline}
The Wiener Processes $(W^{j,\alpha,p})$ are the same as in \eqref{networkmodel}. Clearly the law of $\bar{X}^j$ is the
 same as the law of $\bar{X}^k$, and these laws are the same as the law of $\bar{X}$ in \eqref{meanelec1}.

As we have already noted, we cannot directly use the proofs in \cite{sznitman1991topics} because the interaction terms are not 
uniformly Lipschitz. Instead, we will exploit the fact that $\bar{X}$ is Gaussian, so that we may use Borell's 
Inequality to show that the tails of $\sup_{s\in [0,T]}|\bar{X}^{\alpha,p}_s|$ decay exponentially.
\begin{proof}[Proof of Theorem \ref{theorem3}]
Throughout this proof, when summing over indices, the ranges are always such that $i,j \in [-n,n]$,
 $\alpha,\beta \in [1,P]$, $p\in [-s_\alpha,s_\alpha]$ and $q\in [-s_\beta,s_\beta]$. 
We observe that for any $i\in [-n,n]$, $\alpha\in [1,P]$ and $p\in [-s_\alpha,s_\alpha]$,
\begin{multline*}
X^{i,\alpha,p}_t - \bar{X}^{i,\alpha,p}_t = \int_0^t \bigg[ -\tau^{-1}\left( X^{i,\alpha,p}_s 
- \bar{X}^{i,\alpha,p}_s\right) + \\ (2n+1)^{-1}\sum_{j,\beta,q}\bar{J}^{\alpha\beta}\big( \left(X_s^{i,\alpha,p} - X_s^{j,\beta,q}\right)S(X_s^{j,\beta,q})- E^{\bar{\mu}_s}\big[\left(\bar{X}_s^{i,\alpha,p} - y^{\beta,q}\right)S(y^{\beta,q})\big]\big)\bigg] ds \\
= \int_0^t\bigg( -\tau^{-1}\left( X^{i,\alpha,p}_s - \bar{X}^{i,\alpha,p}_s\right) + \\ (2n+1)^{-1}\sum_{j,\beta,q}\bar{J}^{\alpha\beta}\left[ \left(X_s^{i,\alpha,p} - X_s^{j,\beta,q}\right)S(X_s^{j,\beta,q}) -  \left(\bar{X}_s^{i,\alpha,p} - X_s^{j,\beta,q}\right)S(X_s^{j,\beta,q})\right. \\ \left.+ \left(\bar{X}_s^{i,\alpha,p} - X_s^{j,\beta,q}\right)S(X_s^{j,\beta,q}) - \left(\bar{X}_s^{i,\alpha,p} - \bar{X}_s^{j,\beta,q}\right)S(\bar{X}_s^{j,\beta,q})\right.\\ \left.+ \left(\bar{X}_s^{i,\alpha,p} - \bar{X}_s^{j,\beta,q}\right)S(\bar{X}_s^{j,\beta,q}) - E^{\bar{\mu}_s}\big[\left(\bar{X}_s^{i,\alpha,p} - y^{\beta,q}\right)S(y^{\beta,q})\big]\right] \bigg)ds.
\end{multline*}
For any $g\in\mathcal{C}([0,t],\R)$, let $|g|_t^* = \sup_{s\in [0,t]} |g_s|$. We thus find that
\begin{multline*}
 \left|X^{i,\alpha,p}- \bar{X}^{i,\alpha,p}\right|_T^* \leq \int_0^T\bigg( \tau^{-1} \left|X^{i,\alpha,p}_t - \bar{X}^{i,\alpha,p}_t\right| + \\ (2n+1)^{-1}\sum_{j,\beta,q}\bar{J}^{max}\bigg( \left|X^{i,\alpha,p}_t - \bar{X}^{i,\alpha,p}_t\right| + \left|\bar{X}^{i,\alpha,p}_t\right|\left|S(X_t^{j,\beta,q}) - S(\bar{X}_t^{j,\beta,q})\right| + \\
 \left| X_t^{j,\beta,q}S(X_t^{j,\beta,q}) - \bar{X}_t^{j,\beta,q}S(\bar{X}_t^{j,\beta,q})\right| + \\\left.
 + \left|\left(\bar{X}_t^{i,\alpha,p} - \bar{X}_t^{j,\beta,q}\right)S(\bar{X}_t^{j,\beta,q}) - E^{\bar{\mu}_t}\left[\left(\bar{X}_t^{i,\alpha,p} - y^{\beta,q}\right)S(y^{\beta,q})\right]\right|\bigg)\right) dt.
\end{multline*}
Let $S(X^j) \in \mathcal{C}([0,T),\T)$ be defined as $S(X^j)^{\beta,q}_t := S(X_t^{j,\beta,q})$. Now by Lemma \ref{Lem Globally Lipschitz}, the function $xS(x)$ is globally Lipschitz in $x$. Thus summing over $i$, we find that for some constant $L > 0$,
\begin{multline*}
 \sum_{i}\left|X^{i,\alpha,p}- \bar{X}^{i,\alpha,p}\right|_T^* \leq \int_0^T\bigg( L \sum_{i}\left|X^{i,\alpha,p} - \bar{X}^{i,\alpha,p}\right|_t^* + \\ (2n+1)^{-1}\sum_{i,j,\beta,q}\bar{J}^{max}\big( \left|\bar{X}^{i,\alpha,p}\right|_t^*\left|S(X^{j,\beta,q}) - S(\bar{X}^{j,\beta,q})\right|_t^* + \\
  \left|\left(\bar{X}_t^{i,\alpha,p} - \bar{X}_t^{j,\beta,q}\right)S(\bar{X}_t^{j,\beta,q}) - \mathbb{E}^{\bar{\mu}_t(y)}\left[\left(\bar{X}_t^{i,\alpha,p} - y^{\beta,q}\right)S(y^{\beta,q})\right]\right|\big) \bigg)dt.
\end{multline*}
Write $\norm{X}_T^* = \sup_{\alpha,p} |X^{\alpha,p}|_T^*$. Taking expectations, and using symmetry, we find that
\begin{multline*}
(2n+1)\E\left[\norm{X^{0}- \bar{X}^{0}}_T^*\right] \leq \int_0^T\bigg[ (2n+1)L\E\left[\norm{X^{0} - \bar{X}^{0}}_t^* \right]+ \\ \bar{s}\bar{J}^{max}\sum_{j}\E\left[ \norm{\bar{X}^{0}}_t^*\norm{S(X^{j}) - S(\bar{X}^{j})}_t^*\right]\bigg] dt + B(T).
\end{multline*}
Here, 
\begin{multline*}
 B(T) = \\T\bar{J}^{max}\sup_{t\in [0,T]}\sup_{\alpha,p}\sum_{\beta,q} \E\left[\sum_{j} \left|\left(\bar{X}_t^{i,\alpha,p} - \bar{X}_t^{j,\beta,q}\right)S(\bar{X}_t^{j,\beta,q}) - \mathbb{E}^{\bar{\mu}_t}\left[\left(\bar{X}_t^{i,\alpha,p} - y^{\beta,q}\right)S(y^{\beta,q})\right]\right|\right] \\
 \leq T\bar{J}^{max}\sup_{\alpha,p}\sup_{t\in [0,T]}\sum_{\beta,q}\E\left[\bigg(\sum_j R_t(i,j,\alpha,p,\beta,q)\bigg)^2\right]^{\frac{1}{2}},
\end{multline*}
through Jensen's Inequality, where
\[
R_t(i,j,\alpha,p,\beta,q) = \left(\bar{X}_t^{i,\alpha,p} - \bar{X}_t^{j,\beta,q}\right)S(\bar{X}_t^{j,\beta,q}) - \mathbb{E}^{\bar{\mu}_t}\left[\left(\bar{X}_t^{\alpha,p} - y^{\beta,q}\right)S(y^{\beta,q})\right].
\]
Now if $j\neq k$, $\E\left[ R_t(i,j,\alpha,p,\beta,q)R_t(i,k,\alpha,p,\beta,q)\right] = 0$, since the integrands are independent and of zero mean. By \eqref{eqn norm bound}, $\E\left[ R_t(i,j,\alpha,p,\beta,q)^2\right]$ possesses an upper bound which is uniform for all $t\in [0,T]$ and all indices. This means that there exists a constant $K(T)$ such that $B(T) \leq K(T) (2n+1)^{\frac{1}{2}}$. 

For some $M > 0$, let $A_M = \lbrace X \in \mathcal{C}([0,T],\T) : \norm{X}_T^* \geq M \rbrace$ and
\begin{equation}
\kappa_M = \int_{A_M} \norm{\bar{X}^0}_T^* d\mathbb{P}(\bar{X}^0).\label{defn kappaM}
\end{equation}
We thus find that, since $\norm{S(X^{j}) - S(\bar{X}^{j})}_t^* <1$,
\begin{equation*}
\E\left[ \norm{\bar{X}^{0}}_t^*\norm{S(X^{j}) - S(\bar{X}^{j})}_t^*\right] \leq M\E\left[\norm{X^j - \bar{X}^j}_t^*\right] + \kappa_M.
 \end{equation*}
Therefore
 \begin{multline*}
 \E\left[\norm{X^{0}- \bar{X}^{0}}_T^*\right] \leq \\ \int_0^T\left[ (M\bar{s}\bar{J}^{max}+L)\E\left[\norm{X^{0} - \bar{X}^{0}}_t^* \right]+ \bar{s}\bar{J}^{max}\kappa_M\right] dt + (2n+1)^{-\frac{1}{2}}K(T).
 \end{multline*}
 Through Gronwall's Inequality, since trivially $\E\left[\norm{X^{0}- \bar{X}^{0}}_0^*\right] =0$, 
 \[
  \E\left[\norm{X^{0}- \bar{X}^{0}}_T^*\right] \leq \left((2n+1)^{-\frac{1}{2}}K(T)+T\bar{s}\bar{J}^{max}\kappa_M\right)\exp(T(M\bar{s}\bar{J}^{max}+L)).
 \]
 By Lemma \ref{Lemma kappa M bound}, as $M\to \infty$,
 \begin{equation} \label{eq temporary 15}
 T\bar{s}\bar{J}^{max}\kappa_M\exp(T(M\bar{s}\bar{J}^{max}+L)) \to 0.
 \end{equation}
 Thus by taking $M,n\to \infty$, such that $(2n+1)^{-\frac{1}{2}}\exp(T(M\bar{s}\bar{J}^{max}+L)) \to 0$, 
 \[
 \E\left[\norm{X^{0}- \bar{X}^{0}}_T^*\right] \to 0. 
 \]

\end{proof}
For the following lemma, note that $\kappa_M$ is defined in \eqref{defn kappaM}. This Lemma is needed for \eqref{eq temporary 15} in the previous proof.
\begin{lem}\label{Lemma kappa M bound}
There exist constants $\delta,M_0 > 0$ such that for all $M > M_0$,
\[
\kappa_M \leq \exp\left( - \delta M^2 \right).
\]
\end{lem}
\begin{proof}
Let $\delta = \frac{1}{8}\left(\sup_{s\in [0,T]}\norm{Q(s)}\right)^{-1}$, where $(m,Q)$ is the solution of the mean field equations \eqref{eqn mbar}-\eqref{eqn Pbar}. We saw in the previous section that the law of $\bar{X}$ (which is the same as the law of $\bar{X}^j$) is Gaussian. Moreover the means $m^{\alpha,p}(s)$ and variances $Q^{\alpha\alpha}_{pp}(s)$ are continuous in $s$, and therefore bounded for $s\in [0,T]$. Let $\bar{Y}^{\alpha,p}_s := \bar{X}^{\alpha,p}_s - m^{\alpha,p}(s)$ and $\tilde{Q}^{\alpha}_p = \sup_{s\in [0,T]}Q^{\alpha\alpha}_{pp}(s)$. By Borell's Inequality (see \cite[Theorem 2.1]{adler1990introduction}), $\E\left[\left| \bar{Y}^{\alpha,p} \right|_T^*\right] < \infty$ and for all $\lambda > 0$, 
\begin{equation*}
\mathcal{P}\left(\left| \bar{Y}^{\alpha,p} - \E\left[\left| \bar{Y}^{\alpha,p} \right|_T^*\right] \right|_T^*> \lambda\right) \leq 2\exp\left( - \frac{1}{2\tilde{Q}^\alpha_p}\lambda^2\right).
\end{equation*}
Since $\sup_{s\in [0,T]}\left|m^{\alpha,p}_s\right| < \infty$ for all $\alpha\in [1,P]$ and $p\in [-s_\alpha,s_\alpha]$, there must exist an $M_0$ such that for all $M\geq M_0$
\begin{equation*}
\mathbb{P}\left( \norm{X}_T^* \geq M\right) \leq (M+1)^{-1}\exp\left( -2\delta M^2\right).
\end{equation*}
Hence
\begin{multline*}
\kappa_M \leq \sum_{k=0}^\infty (M+k+1) \mathbb{P}\bigg( \norm{X}_T^* \geq (M+k)\bigg)
\leq \sum_{k=0}^\infty\exp\left( -2\delta(M+k)^2\right)
\leq \exp( -\delta M^2),
\end{multline*}
for $M_0$ sufficiently large. 
\end{proof}
\section{Numerical simulation of the mean-field equation and discussion}\label{numeric}
In this section we perform a numerical simulation of the mean-field equations \eqref{meanelec1}.
 
We consider two populations of neurons, with the first one denoted by $\alpha$  and the second
 one denoted by $\beta$. The ratio of neurons in the populations is fixed at 1:1, with $s_1 = s_2 = 0$.
 We take the initial condition $X_{ini} = (0.2,-0.35)$, the currents are fixed at $I^\alpha = 0.2$, $I^{\beta}=-0.2$,
 the noise intensity is $\sigma^\alpha = \sigma^\beta = 0.2$, and the sigmoid function is $S(x) = (1+\exp(-x))^{-1}$. 
The synaptic weights are set to
 $J^{\alpha\alpha} = -0.11, J^{\alpha\beta} = -1.1, J^{\beta\alpha} = 0.44$ and $J^{\beta\beta} = -0.11$.
 The results are displayed in Figure 1.
\begin{figure} 
\includegraphics[width=12cm,height=6cm]{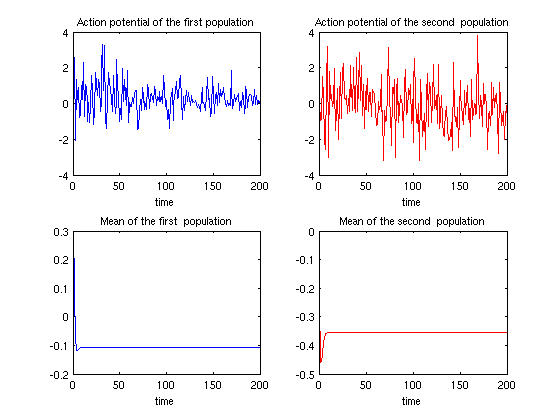}
\label{figure1}
\caption{In the bottom two figures we plot the evolution of the means $m^\alpha$ and $m^\beta$ of the mean-field equations. 
The covariances $Q^{\alpha\alpha}_{00}$ and $Q^{\beta\beta}_{00}$ are not plotted but display a similar
 asymptotic convergence in time. In the top two figures we plot Monte-Carlo simulations of $\bar{X}^{\alpha,0}$ and $\bar{X}^{\beta,0}$.}
\end{figure}

As we noted in the introduction, mean-field 
equations are useful because they reduce the dynamics of large-scale ensembles of neurons to a scale which is more easily
 studied. We have seen that the mean-field equations for our system reduce to a system of ordinary differential equations in
 the mean and covariance. One can therefore apply dynamical systems techniques (in particular, bifurcation analysis)
 to study the properties of the mean-field equations. The above example demonstrates an ergodic convergence in time. 
 However one could straightforwardly apply the theory in this paper to investigate more complicated examples. For example, 
one could search for two-population activator-inhibitor systems where the means and covariances oscillate in time 
(a well-known result). One could also investigate systems of three or more populations of neurons and search for other types of bifurcations. We plan to perform a detailed bifurcation analysis of the mean-field equations in future work.

\end{document}